\documentclass[11pt,oneside,onecolumn,reqno]{amsart}

\usepackage{amssymb,latexsym,
amsxtra,amscd,amsfonts,amsthm,amsmath,verbatim,epsfig}
\usepackage{xypic}

\numberwithin{equation}{section}

\def\image{\operatorname{image}}
\def\ker{\operatorname{ker}}
\def\dim{\operatorname{dim}}

\def\height{\operatorname{ht}}
\def\grade{\operatorname{grade}}
\def\ass{\operatorname{Ass}}

\def\spec{\operatorname{Spec}}
\def\ZZ{\mathbb Z}
\newcommand{\m}{\mathfrak m}
\newcommand{\ov}{\overline}
\newcommand{\lm}{{\lambda}}
\theoremstyle{plain}
\newtheorem{theorem}[equation]{Theorem}
\newtheorem{corollary}[equation]{Corollary}

\newtheorem{lemma}[equation]{Lemma}

\theoremstyle{definition}

\newtheorem{example}[equation]{Example}
\newtheorem{definition}[equation]{Definition}

\theoremstyle{remark}
\newtheorem{remark}[equation]{Remark}
\begin{document}
\title[ ] {Local Cohomology of Bigraded Rees Algebras and Normal Hilbert
Coefficients}
\author[]{Shreedevi K. Masuti and J. K. Verma}
\thanks{The first author is supported by the National Board for Higher
Mathematics, India}
\thanks{Key words : Normal Hilbert polynomial of two ideals, 
analytically unramified local ring, integral closure of an ideal, 
good joint reductions, local cohomology of Rees algebra, 
normal joint reduction number}
\address{Department of Mathematics, Indian Institute of Technology
Bombay, Powai, Mumbai, India - 400076}
\email{shreedevi@math.iitb.ac.in}
\email{jkv@math.iitb.ac.in}
\maketitle
\begin{abstract} Let $(R,\m)$ be an analytically unramified 
Cohen-Macaulay local ring of dimension $2$ with infinite residue 
field and  $\ov{I}$ be the integral closure of an ideal $I$ in $R$.
Necessary and sufficient conditions are given for 
$\ov{I^{r+1}J^{s+1}}=a\ov{I^rJ^{s+1}}+b\ov{I^{r+1}J^s}$ to hold 
$ \mbox{ for all }r \geq r_0 \mbox{ and }s \geq s_0$ in terms of 
vanishing of $[H^2_{(at_1,bt_2)}(\ov{\mathcal{R}^\prime}(I,J))]_{(r_0,s_0)}$, 
where $a \in I,b \in J$ is a good joint reduction of the filtration $\{\ov{I^rJ^s}\}.$
This is used to derive a theorem due to Rees on normal 
joint reduction number zero. The  vanishing of 
$\ov{e}_2(IJ)$ is shown to be equivalent to Cohen-Macaulayness of 
$\ov{\mathcal{R}}(I,J)$.
\end{abstract}
\section{Introduction}
\noindent Let $R$ be a commutative ring. Let $I$ 
be an ideal of $R$. An element  $x \in R$ is called integral over 
$I$ if $x$ satisfies the equation  
$$x^n+a_1x^{n-1}+\cdots+a_n=0$$ for some $a_i\in I^i, 
i=1,2,\ldots,n.$ The set $\ov{I}$ of elements that are integral 
over $I$ is an ideal, called the {\it integral closure} of $I$. 
If $I=\ov{I}$ then $I$ is called complete or integrally closed.
A  Noetherian local ring $(R,\m)$ is said to be 
{\it analytically unramified}  if its $\m-$adic completion is 
reduced. Let $I$ be an $\m-$primary ideal. In an analytically 
unramified local ring $R$ of dimension $d$, there exists a  polynomial 
$\ov{P}_I(x)\in \mathbb{Q}[x]$ of degree $d$ called the  
{\it normal Hilbert polynomial} of $I$, such that 
$$\lm(R/\ov{I^n})=\ov{P}_I(n) \mbox{ for }n \gg 0,$$
where $\lm(M) $ denotes the length of an $R$-module $M,$ 
\cite[Theorem 1.4]{rees1} and \cite[Theorem 1.1]{rees4}. We write 
\begin{eqnarray*}
 \ov{P}_I(n+1)=\ov{e}_0(I)\binom{n+d}{d}-\ov{e}_1(I)
\binom{n+d-1}{d-1}+\cdots+(-1)^d\ov{e}_d(I)
\end{eqnarray*}
for some integers $\ov{e}_i(I)$ $i=0,\ldots,d$. The coefficient 
$\ov{e}_0(I)=e(I)$ is called the {\it multiplicity of $ I.$} 
P. B. Bhattacharya \cite[Theorem 8]{b} showed that for $\m-$primary
ideals $I$ and $J$ in a Noetherian local ring $(R,\m)$ of dimension
$d$  there exist integers $e_{(i,j)}(I,J)$ such that for 
large $r,s$ 
\begin{eqnarray*}
 \lm(R/I^rJ^s)=\sum_{i+j \leq d}(-1)^{d-(i+j)}e_{(i,j)}(I,J)
 \binom{r+i-1}{i}\binom{s+j-1}{j}.
\end{eqnarray*}

\noindent Let $\mathbb N$ denote the set of nonnegative integers. 
Rees studied  the numerical function 
$$\ov{H}_{I,J}:\mathbb{N}^2 \longrightarrow \mathbb{N} 
\mbox{ defined as } \ov{H}_{I,J}(r,s)=\lm(R/\ov{I^rJ^s}).$$ He 
 proved  \cite{rees} that there exists a polynomial 
$\ov{P}_{I,J}(x,y) \in \mathbb{Q}[x,y]$ 
of total degree $d$ such that $\ov{P}_{I,J}(r,s)=\ov{H}_{I,J}(r,s)$ for $r,s \gg
0$ in an analytically unramified local ring. We write 
\begin{eqnarray*}
 \ov{P}_{I,J}(x,y)=\sum_{i+j \leq
d}(-1)^{d-(i+j)}\ov{e}_{(i,j)}(I,J)\binom{x+i-1}{i}\binom{y+j-1}{j}
\end{eqnarray*} 
for some integers $\ov{e}_{(i,j)}(I,J)$. Let 
\begin{eqnarray*}
\mbox{ extended Rees ring of } \mathcal I=\{I^rJ^s\}\mbox{ be }
{\mathcal R}^\prime(I,J)=\bigoplus_{r,s \in \mathbb{Z}} 
I^rJ^st_1^rt_2^s,\\
\mbox{ extended Rees ring of }\mathcal{I}=\{\ov{I^rJ^s}\} 
\mbox{ be }
\ov{{\mathcal R}^\prime}(I,J)=\bigoplus_{r,s \in \mathbb{Z}}
\ov{I^rJ^s}t_1^rt_2^s. 
\end{eqnarray*}
One of the main objectives of this paper is to understand an 
interesting 
theorem of  Rees \cite{rees} which asserts that  
$\ov{e}_2(IJ)=\ov{e}_2(I)+\ov{e}_2(J)$ for $\m-$primary ideals 
$I$ and $J$ in an analytically unramified 
Cohen-Macaulay local ring of dimension $2$ with infinite residue field if and only if for all $r,s \geq 0,$ 
\begin{equation} \label{eq3}
 \ov{I^{r+1}J^{s+1}}=a\ov{I^{r}J^{s+1}}+b\ov{I^{r+1}J^{s}},
\end{equation}
where $a \in I,b \in J$ is a good joint reduction of $I$ and $J$. See section $2$. 
As a consequence Rees proved that product of complete ideals is 
complete in $2-$dimensional pseudo-rational local rings. 
Rees showed that regular local rings are pseudo-rational and thus 
generalized Zariski's product theorem. Another consequence of Rees's
theorem is a formula for the Hilbert polynomial of an integrally
closed ideal in a two dimensional regular local ring. 
We generalize Rees's  theorem:
\begin{theorem} \label{thm11} Let $R$ be an analytically 
unramified Cohen-Macaulay local ring of
dimension $2$ with infinite residue field and $(a,b)$ be a good joint reduction of 
the filtration $\{\ov{I^rJ^s}\}$. Let $r_0,s_0 \geq 0$. Then following statements are equivalent:\\
$(1)\;
\ov{e}_2(I)+\ov{e}_2(J)-\ov{e}_2(IJ)=\lm(R/\ov{I^{r_0}J^{s_0}})-g_{r_0}(I,J)-h_{
s_0}(I,J)-r_0s_0e(I|J),$\\
$(2)\; [H^2_{(at_1,bt_2)}(\ov{\mathcal{R}^\prime}(I,J)]_{(r_0,s_0)}=0,\\$
$(3)\; \ov{I^{r+1}J^{s+1}}=a\ov{I^rJ^{s+1}}+b\ov{I^{r+1}J^s} \mbox{ for }r \geq
r_0,~s\geq s_0,$\\
where $e(I|J)=e_{(1,1)}(I,J)$ and $g_{r_0}(I,J),h_{s_0}(I,J)$ satisfy
\begin{eqnarray*}
 \lm(\ov{J^s}/\ov{I^{r_0}J^s})&=&e(I|J)r_0s+g_{r_0}(I,J) \mbox{ for }s \gg
0\mbox{ and }\\
\lm(\ov{I^r}/\ov{I^rJ^{s_0}})&=&e(I|J)rs_0+h_{s_0}(I,J) \mbox{ for }r \gg 0 .
\end{eqnarray*}
\end{theorem}

\noindent We derive a formula for $\lm([H^2_{(at_1,bt_2)}
(\ov{\mathcal{R}^\prime}(I,J))]_{(r,s)})$ in terms of 
Hilbert coefficients. The above theorem also gives a cohomological
interpretation of Rees's theorem since   
$$\lm([H^2_{(at_1,bt_2)}(\ov{\mathcal{R}^\prime}(I,J))]_{(0,0)})
=\ov{e}_2(I)+\ov{e}_2(J)-\ov{e}_2(IJ).$$

We will gather some preliminary results about existence of 
good joint reductions in Section 2. 

In Section 3, we calculate
the local cohomology of bigraded extended Rees algebra of 
the filtration $\{\ov{I^rJ^s}\}.$ 

In Section 4,  a new proof of 
Rees's theorem and its generalization are obtained. 
We give an application of Theorem \ref{thm11} to the normal 
reduction number of an ideal by deriving a  result of T. Marley 
\cite[Corollary 3.8]{marleythesis} which asserts 
that $\ov{r}(I) \leq k+1$ if and only if $\lm(R/\ov{I^k})
=\ov{P}_I(k)$.

In section 5, we study vanishing of $\ov{e}_2(IJ)$. 
We prove that  the vanishing of $\ov{e}_2(IJ)$ is equivalent to 
Cohen-Macaulayness of $\ov{\mathcal{R}}(I,J)$, where 
$$\displaystyle{\ov{\mathcal{R}}(I,J)=\bigoplus_{r,s \geq 0}
\ov{I^rJ^s}t_1^rt_2^s}=\mbox{ the  Rees ring of the filtration }
\mathcal{I}=\{\ov{I^rJ^s}\}.$$
 
We refer \cite{bruns-herzog} for all undefined terms.\\
{\bf Acknowledgement:} We thank Manoj Kummini for many clarifications and 
helpful conversations
which improved the paper.

\section{Preliminary Results}
\noindent Let $(R,\m)$ be a Noetherian local ring 
of dimension $2$ and $I$ and $J$ be $\m$-primary ideals in $R.$ 
D. Rees   introduced joint reductions in  \cite{rees3} for the 
filtration $\{I^rJ^s\}$. A sequence $(a,b)$ such that $a \in I$ and
$b \in J$ is called a {\it joint reduction} of the sequence of ideals 
$(I, J)$ if for all large $r,s$ we have 
$$I^rJ^s=aI^{r-1}J^s+bI^rJ^{s-1}.$$

\noindent A sequence
$(a,b)$ is a {\it joint reduction} of the filtration 
$\mathcal{I}=\{\ov{I^rJ^s}\}$
if $a \in I$ and $b \in J$ and for 
$r,s \gg 0$
$$\ov{I^rJ^s}=a\ov{I^{r-1}J^s}+b\ov{I^rJ^{s-1}}.$$ 
We say $(a,b)$ is a {\it good joint reduction} of 
$\mathcal{I}$ if $(a,b)$ is a joint reduction of $\mathcal{I}$ 
and
\begin{eqnarray*}
 (a) \cap \ov{I^rJ^s}&=&a\ov{I^{r-1}J^s} \mbox{ for } s \in \ZZ 
\mbox{ and }r>0  \mbox{ and } \\
(b) \cap \ov{I^rJ^s} &=& b\ov{I^rJ^{s-1}} \mbox{ for }  
r \in \ZZ \mbox{ and }  s>0.
\end{eqnarray*}

\noindent Existence of joint reductions of the sequence $(I,J)$ in a
Noetherian local ring with infinite residue field was 
proved by Rees in \cite{rees3}. In this section we prove existence of good joint
reductions of $\mathcal I=\{\ov{I^rJ^s}\}$ in an 
analytically unramified Cohen-Macaulay local ring of dimension $2$ with infinite
residue field. We need an analogue of Rees's Lemma for
$\mathcal{I}=\{\ov{I^rJ^s}\}$ 
\cite[Lemma 1.2]{rees3}.

\begin{lemma} \em[Rees's Lemma]\label{r1.2}
 Let $(R,\m)$ be an analytically unramified local ring of dimension $d$ with
infinite residue field and $I,J$ be ideals in $R.$ 
Let $S$ be a finite set of prime ideals of $R$ not containing $IJ.$ Then there
exist $a \in I, b \in J$ not contained in any of 
the prime ideals in $S$ and integers $r_0,s_0$ such that 
\begin{eqnarray}\label{eq4}
(a) \cap \ov{I^rJ^s}&=&a\ov{I^{r-1}J^s} \mbox{ for any integer }s \mbox{ and }r
\geq r_0 \mbox{ and }  \\
\label{eq5}(b) \cap \ov{I^rJ^s} &=& b\ov{I^rJ^{s-1}} \mbox{ for any integer } r
\mbox{ and } s \geq s_0.
\end{eqnarray}
\end{lemma}

\noindent Let the extended associated graded ring with respect to $I $ of 
$$\mathcal I =\{I^rJ^s\}\mbox{ be } \mathcal{R}^\prime(I,J|I)=\bigoplus_{r,s\in
\mathbb{Z}}I^rJ^s/I^{r+1}J^s \mbox{ and }$$ 
$$\mathcal{I}=\{\ov{I^rJ^s}\} \mbox{ be }
\ov{\mathcal{R}^\prime}(I,J|I)=\bigoplus_{r,s\in 
\mathbb{Z}}\ov{I^rJ^s}/\ov{I^{r+1}J^s}.$$
Similarly extended associated graded ring with respect to $J$ of  
\begin{eqnarray*}
 \mathcal I &=&\{I^rJ^s\}\mbox{ be } \mathcal{R}^\prime(I,J|J)=\bigoplus_{r,s\in
\mathbb{Z}}I^rJ^s/I^{r}J^{s+1} \mbox{ and } \\
\mathcal{I}&=&\{\ov{I^rJ^s}\} \mbox{ be }
\ov{\mathcal{R}^\prime}(I,J|J)=\bigoplus_{r,s\in
\mathbb{Z}}\ov{I^rJ^s}/\ov{I^{r}J^{s+1}}.
\end{eqnarray*}

\noindent For a bigraded $B_{(0,0)}$-algebra $B=\bigoplus_{r,s \in \mathbb Z} B_{(r,s)}$, 
let $B_{1}^+$ denote the ideal generated by $\bigoplus_{r \geq 1} B_{(r,s)}$, 
$B_2^+$ denote the ideal generated by $\bigoplus_{s \geq 1} B_{(r,s)}$. 
To prove equality for all $r>0$ in equation \ref{eq4} and all $s>0$ in equation \ref{eq5} we show that
$\grade(\ov{\mathcal{R}^\prime}(I,J|I)_1^+)$ and $\grade(\ov{\mathcal{R}^\prime}(I,J|J)_2^+) 
\geq 1 $. In order to prove this we need following lemmas. 
We put $ \mathcal{R}^\prime=\mathcal{R}^\prime(I,J),\ov{\mathcal{R}^\prime}=
\ov{\mathcal{R}^\prime}(I,J), u_i=t_i^{-1}$ for $i=1,2$ and $u=u_1u_2$.

\begin{lemma} \label{p}
Let $(R,\m)$ be an analytically unramified Cohen-Macaulay local ring of 
dimension $2$ and $I,J$ be $\m-$primary ideals. 
Then 
$$\height (u_1,\ov{\mathcal R^\prime}_1^+)=\height (u_2,\ov{\mathcal R^\prime}_2^+) =2.$$
 \end{lemma}
\begin{proof}
We have 
\begin{eqnarray*}
 \ov{\mathcal R^\prime}(I,J)/(u_1,\ov{\mathcal R^\prime}_1^+) \cong \bigoplus_{n \in \mathbb Z}\ov{J^n}/\ov{IJ^n}.
\end{eqnarray*}
Let $\ov{\mathcal R^\prime}(J)=\bigoplus_{n \in \mathbb Z}\ov{J^n}t_2^n$ be the extended 
Rees ring of $\{\ov{J^n}\}$. Since any minimal prime ideal containing $u_2$ in $\ov{\mathcal R^\prime}(J)$ 
contains $\m$ $$\height \m\ov{\mathcal R^\prime}(J)=\height I\ov{\mathcal R^\prime }(J)=1.$$ 
Since $I\ov{J^n}$ is a reduction of $\ov{IJ^n}$, $I\ov{\mathcal R^\prime}(J)$ and 
$(\oplus_{n \in \mathbb Z}\ov{IJ^n}t_2^n)$ have 
same radical. Therefore $\height (\oplus_{n \in \mathbb Z}\ov{IJ^n}t_2^n)=1 $. Hence 
$$\dim (\bigoplus_{n \in \mathbb Z}\ov{J^n}/\ov{IJ^n})=2.$$ Thus 
$\height (u_1,\ov{\mathcal R^\prime}_1^+) =2$. \\
Similar argument shows that $\height (u_2,\ov{\mathcal R^\prime}_2^+) =2$.
\end{proof}

\begin{lemma} \em{\cite[Lemma 3.24]{marleythesis}}\label{o}
Let $R$ be a Noetherian ring and $x \in R$ be a nonzerodivisor such that $(x)$
is integrally closed ideal. Then $\height P=1$ for all 
associated primes of $(x).$ 
\end{lemma}

\begin{lemma} \label{lemma1}
 Let $A \subseteq B$ be a ring extension and $C$ be the integral closure of $A$ in
$B.$ Let $z \in C$ be a nonzerodivisor in $C$ such that $z$ is a unit 
in $B.$ Then $zC$ is an integrally closed ideal.
\end{lemma}
\begin{proof}
 Let $x \in C$ be integral over $zC.$ Then 
$$x^n+zc_1x^{n-1}+\cdots+z^{n-1}c_{n-1}x+z^nc_n=0$$ 
for some $c_i \in C.$ Thus 
$$\left(\frac{x}{z}\right)^n+c_1\left(\frac{x}{z}\right)^{n-1}+\cdots+c_{n-1}
\left(\frac{x}{z}\right)+c_n=0.$$
Hence $xz^{-1}\in B$ is integral over $C$. Therefore $x \in zC.$
\end{proof}

\begin{theorem}
 Let $(R,\m)$ be an analytically unramified Cohen-Macaulay local ring of dimension $2$ with 
infinite residue field. Then $\grade(\ov{\mathcal{R}^\prime}(I,J|I)_{1}^+) $ and 
$\grade(\ov{\mathcal{R}^\prime}(I,J|J)_{2}^+)$ are positive.
\end{theorem}
\begin{proof}
By Lemma \ref{lemma1}, $u_1\ov{\mathcal{R}^\prime}$ is an 
integrally closed
ideal. Therefore  by Lemma \ref{o}, all the associated primes of 
$u_1\ov{\mathcal{R}^\prime}$ have height $1$. Since $\height (u_1,\ov{\mathcal
R^\prime}_{1}^+)=2 >1,$ $(u_1,\ov{\mathcal R^\prime}_{1}^+) $ is not contained in any of the 
associated primes of $u_1\ov{\mathcal{R}^\prime}$. Since 
$$\ov{\mathcal R^\prime} /(u_1) \cong \ov{\mathcal R^\prime}(I,J|I), $$
$\grade(\ov{\mathcal{R}^\prime}(I,J|I)_{1}^+) \geq 1$. \\
Similar argument shows that $ \grade(\ov{\mathcal{R}^\prime}(I,J|J)_{2}^+)\geq 1$.\\
\end{proof}

\begin{lemma} \label{l}
Let $(R,\m)$ be an analytically unramified Cohen-Macaulay local ring of dimension $2$ 
with infinite residue field. Let $I,J$ be $\m-$primary ideals 
and let $a \in I\setminus I^2, b \in J\setminus
J^2$ be nonzerodivisors such that 
\begin{eqnarray*} 
(a) \cap \ov{I^rJ^s}&=&a\ov{I^{r-1}J^s} \mbox{ for }r \gg 0 \mbox{ and all integers } s \mbox{ and } \\
(b) \cap \ov{I^rJ^s}&=&b\ov{I^{r}J^{s-1}} \mbox{ for all integers }r \mbox{ and }s \gg0 .
\end{eqnarray*}
Then 
\begin{eqnarray*}
(a) \cap \ov{I^rJ^s}&=&a\ov{I^{r-1}J^s} \mbox{ for all }r>0 \mbox{ and all
integers }s \mbox{ and } \\
(b) \cap \ov{I^rJ^s}&=&b\ov{I^{r}J^{s-1}} \mbox{ for all integers }r \mbox{ and } s>0.
\end{eqnarray*}
\end{lemma}
\begin{proof}
Let $a^*$ denote the image of $at_1$ in $[\mathcal{R}^\prime(I,J| I)]_{(1,0)}$. 
We show that $(0:_{\ov{\mathcal R^\prime}(I,J|I)}a^*) 
=0$. Fix $p,q \in \mathbb{Z}$. Let $z^* \in
(0:_{\ov{\mathcal{R}^\prime}(I,J|I)}a^*)\cap
[\ov{\mathcal{R}^\prime}(I,J|I)]_{(p,q)}$ for some $z \in \ov{I^pJ^q}$. Then for any $y^* \in 
[\ov{\mathcal{R}^\prime}(I,J|I)]_{(u,v)}$, where $y \in \ov{I^uJ^v}$, 
$y^*z^*a^*=0$. Hence $yza \in (a) \cap \ov{I^{u+p+2}J^{v+q+1}}=a \ov{I^{u+p+1}J^{v+q+1}}$ for 
$u \gg 0$. Therefore $y^*z^*=0$ for $u \gg 0$. Thus
$(\ov{\mathcal{R}^\prime}(I,J|I)_{1}^+)^nz^*=0$ for some $n.$ Hence 
$z^* \in H^0_{\ov{\mathcal{R}^\prime}(I,J|I)_{1}^+}(\ov{\mathcal{R}^\prime}(I,J|I))=0.$
Thus $(0:_{\ov{\mathcal{R}^\prime}(I,J|I)}a^*)=0.$ \\
Let $r >0$ and $y=az \in (a) \cap \ov{I^rJ^s}$ for some $0 \neq z \in R$. Since 
$\bigcap_k \ov{I^{k}J^{s}}=0$, there exists $k$ such that $z \in \ov{I^{k}J^{s}}
\setminus \ov{I^{k+1}J^{s}}$. Suppose $k <r-1$. 
Then $0\neq z^*\in[{\ov{\mathcal{R}^\prime}(I,J|I)}]_{(k,s)}$ and $a^*z^*=0$, 
a contradiction. Therefore $z \in \ov{I^{r-1}J^s}.$ Hence $y\in a\ov{I^{r-1}J^s}$.
Similar argument gives 
$$(b) \cap \ov{I^rJ^s}=b\ov{I^{r}J^{s-1}} \mbox{ for any integer }r \mbox{ and } 
s>0.$$
\end{proof}

\begin{theorem} \em{\cite{rees}} \label{reqAB}Let $(R,\m)$ be an analytically
unramified Cohen-Macaulay local ring of dimension $2$ 
with infinite residue field and let $I,J$ be $\m$-primary ideals. Then there
exists a good joint reduction $(a,b)$ of $\{\ov{I^rJ^s}\}$.
\end{theorem}
\begin{proof}
\noindent Let $S_1$ be a set of associated primes of $R.$ Since $\m \notin S_1$,
$IJ \nsubseteq \mathfrak{p}$ for any $\mathfrak{p} \in 
S_1.$ Therefore by Lemma \ref{r1.2} and \ref{l} there exist $a \in I\setminus
I^2,b_1 \in J\setminus J^2$ not contained in any of the prime ideals of $S_1$
such that 
\begin{eqnarray*}
(a) \cap \ov{I^rJ^s}&=&a\ov{I^{r-1}J^s} \mbox{ for all }r>0 \mbox{ and }s \in \mathbb{Z} \\
(b_1) \cap \ov{I^rJ^s}&=&b_1\ov{I^{r}J^{s-1}} \mbox{ for all }r \in \mathbb{Z} \mbox{ and } s>0.
\end{eqnarray*}
Let $y=ab_1$ and $R^\prime=R/(y)$. Then for all $r,s>0$,
\begin{eqnarray*}
 (y) \cap \ov{I^rJ^s}&=&(y) \cap ((a) \cap \ov{I^rJ^s})\\
&=&(y) \cap a\ov{I^{r-1}J^s}\\
&=&a((b_1) \cap \ov{I^{r-1}J^{s}})\\
&=&y\ov{I^{r-1}J^{s-1}}.
\end{eqnarray*}
Let $S_2=\{\mathfrak{p} \in \spec R \mid  \mathfrak{p}/(y) \in \ass R^\prime \}$
and $\mathcal{I}=\{\mathcal{I}_{(r,s)}=
\frac{\ov{I^rJ^s}+(y)}{(y)}\}$ be a filtration in $R^\prime$ and
$L=\bigoplus_{r,s \in \mathbb{Z}} \mathcal{I}_{(r,s)}t_1^rt_2^s.$ Let 
\begin{eqnarray*}
S_{21} &=&\{\mathfrak{p} \in \ass_{\ov{\mathcal R^\prime}}(\ov{\mathcal{R}^\prime}/ u\ov{\mathcal
R^\prime}) \mid J t_2 \nsubseteq \mathfrak{p}\}\\
S_{22} &=&\{\mathfrak{p} \in \ass_L(L/ u L) \mid J^\prime t_2 \nsubseteq
\mathfrak{p}\}
\end{eqnarray*} 
where $'$ denotes image of an ideal in $R^\prime.$
Then $\mathfrak{p} \cap J+\m J/\m J$ is a proper subspace of $J/\m J$ for
$\mathfrak{p} \in S_2$. Similarly $q \cap Jt_2+\m J
t_2/\m Jt_2$ (resp. $q \cap J^\prime t_2+\m^\prime J^\prime t_2/\m^\prime
J^\prime t_2$) is a proper subspace of $Jt_2 /\m Jt_2\cong
 J/\m J$ (resp. $J^\prime t_2/ \m^\prime J^\prime t_2 \cong J/\m J$) for $q \in
S_{21}$ (resp. $q \in S_{22}$). Since $R/\m$ is 
infinite there exists $b \in J\setminus J^2$ such that $b \notin \mathfrak{p}$
for any $\mathfrak{p} \in S_2$, $bt_2 \notin q$ 
for $q \in S_{21} $ and $b^\prime t_2 \notin Q$ for any $Q \in S_{22}$. Then for any integer $r$ and $s>0$, 
\begin{eqnarray*} 
 (b) \cap \ov{I^rJ^s}&=& b\ov{I^{r}J^{s-1}} \mbox{ in }R . 
\end{eqnarray*} Also for $r,s \gg 0,$
\begin{eqnarray}\label{eq8}
(b^\prime) \cap \mathcal{I}_{(r,s)}&=& b^\prime \mathcal{I}_{(r,s-1)} \mbox{ in } R^\prime.
\end{eqnarray}
See \cite[Lemma 1.2]{rees3} and Lemma \ref{l}. Let $uy+vb \in \ov{I^rJ^s}$ for some $u,v \in R$. 
Let $'$ denotes image of an element in $R^\prime$. Then by equation 
\ref{eq8}, $v^\prime b^\prime\in b^\prime \mathcal{I}_{(r,s-1)}$ for $r,s 
\gg0$. Therefore $vb-wb = yz$ for some $w \in \ov{I^rJ^{s-1}} $ and $z \in R$. 
Thus 
$$uy+vb=wb+y(u+z) \in b \ov{I^rJ^{s-1}}+(y) \cap \ov{I^rJ^s}.$$ 
Therefore for $r,s\gg0,$
\begin{eqnarray*}
(y,b) \cap \ov{I^rJ^s} &=& b\ov{I^rJ^{s-1}}+(y) \cap \ov{I^rJ^s}\\
&=&b\ov{I^rJ^{s-1}}+y\ov{I^{r-1}J^{s-1}}\\
&\subseteq &b\ov{I^rJ^{s-1}}+a\ov{I^{r-1}J^s}
\end{eqnarray*}
Since $R/(y,b)$ is artinian $I^rJ^s \subseteq (y,b)$ for $r,s \gg 0$. Since $R$
is analytically unramified, there exists $h \geq 0$ such 
that $\ov{I^rJ^s} \subseteq I^{r-h}J^{s-h}$ for all integers $r,s \geq h$. Thus
$\ov{I^rJ^s} \subseteq (y,b)$ for $r,s \gg 0$. Hence for $r,s \gg 0,$
$$\ov{I^rJ^s}=a\ov{I^{r-1}J^s}+b\ov{I^rJ^{s-1}}.$$
\end{proof}

\section{Local Cohomology of Bigraded Rees Algebras}
\noindent Let $(R,\m)$ be a Noetherian local ring of dimension $2$. Let $I,J$ be
$\m$-primary ideals and $(a,b)$ be a good joint reduction 
of $\mathcal I=\{\ov{I^rJ^s}\}.$ In what follows we
derive a formula for 
local cohomology of $\ov{\mathcal{R}^\prime}$ with support in $(at_1,bt_2)$. We show 
that $\lm([H^2_{(at_1,bt_2)}(\ov{\mathcal{R}^\prime})]_{(r,s)})$ 
is independent of good joint reduction for all $r,s \geq 0$.  \\
Consider the Koszul complex on $((at_1)^k,(bt_2)^k)$:

$$
{F^k}^. : 0 \longrightarrow \ov{\mathcal{R}^\prime}
\buildrel\alpha_k\over\longrightarrow \ov{\mathcal{R}^\prime}(k,0) \oplus
\ov{\mathcal{R}^\prime}(0,k)
\buildrel\beta_k\over\longrightarrow \ov{\mathcal{R}^\prime}(k,k)
\longrightarrow
0,
$$
where the maps are defined as,
$$
\alpha_k(1) = ((at_1)^k, (bt_2)^k)~~ ~~ \mbox{ and } \beta_k(u, v) =
-(bt_2)^ku + (at_1)^kv.
$$
The twists are given so that $\alpha_k$ and $\beta_k$ are degree zero
maps.  We have the following commutative diagram of Koszul complexes,

$$
\CD
0@>   >> \ov{\mathcal{R}^\prime} @> \alpha_{k}  >> \ov{\mathcal{R}^\prime}(k,0)
\oplus \ov{\mathcal{R}^\prime}(0,k) @> \beta_{k} >>\ov{\mathcal{R}^\prime}(k,k)
@>  >> 0 \\
@.     @Vf_kVV             @Vg_kVV                @Vh_kVV   @.  \\
0@>   >> \ov{\mathcal{R}^\prime} @> \alpha_{k+1}  >>
\ov{\mathcal{R}^\prime}(k+1,0) \oplus \ov{\mathcal{R}^\prime}(0,k+1) @> 
\beta_{k+1} >>\ov{\mathcal{R}^\prime}(k+1,k+1)@> >> 0
\endCD
$$
where 
$$f_k(1)=1,~~g_k(1,0)=(at_1,0),g_k(0,1)=(0,bt_2) \mbox{ and } h_k(1)=abt_1t_2.$$
Then
\begin{equation} \label{eq1}
H^i_{(at_1,bt_2)}(\ov{\mathcal{R}^\prime}) =
\displaystyle\lim_{\stackrel{\longrightarrow}{k}}
H^i({F^k}^{.})\end{equation}
for all
$i$ by  \cite[Theorem 5.2.9]{bs}.
To obtain an expression for the second local cohomology module of $\ov{\mathcal
R^\prime}$ with support in $(at_1,bt_2)$ we recall 
a few results. From now onwards we assume that $(R,\m)$ is an analytically
unramified Cohen-Macaulay local ring of dimension $2$ with 
infinite residue field. 

For $r,s \gg 0$ we have 
\begin{eqnarray*}
 \lm(R/\ov{I^rJ^s})=\sum_{i+j \leq
2}(-1)^{2-(i+j)}\ov{e}_{(i,j)}\binom{r+i-1}{i}\binom{s+j-1}{j}, 
\end{eqnarray*} 
where $\ov{e}_{(i,j)}=\ov{e}_{(i,j)}(I,J)$.

\noindent We recall the following theorems.
\begin{theorem} \label{thm1} 
 \begin{enumerate}
\item \cite{rees} $\ov{e}_{(2,0)}=e(I),\ov{e}_{(0,2)}=e(J)$ and
$\ov{e}_{(1,1)}=e_{(1,1)}(I,J)$.
 \item \cite[Theorem 1.2]{rees} $\ov{e}_{(1,0)}=\ov{e}_1(I)$ and
$\ov{e}_{(0,1)}=\ov{e}_1(J).$
\item \cite[Lemma 1.1]{rees} Let $s$ be a fixed integer. Then there exists an
integer $r_0$, depending on $s$, such that, for $r \geq r_0$,
$$\lm(\ov{I^r}/\ov{I^rJ^s})=e(I|J)rs+h_s(I,J)$$
where $h_s(I,J)$ is independent of $r.$
\item Let $r$ be a fixed integer. Then there exists an integer $s_0$, depending
on $r$, such that, for $s \geq s_0$,
$$\lm(\ov{J^s}/\ov{I^rJ^s})=e(I|J)rs+g_r(I,J)$$
where $g_r(I,J)$ is independent of $s.$
\end{enumerate}
\end{theorem}

\begin{remark}
Since $\ov{P}_{I,J}(n,n)=\ov{P}_{IJ}(n)$ for all $n$, we also have
$$\ov{e}_{(0,0)}=\ov{e}_2(IJ).$$
\end{remark}

\begin{lemma} \label{lemma2} Let $r,s \geq 0$. Then for all $k,l \geq 1$ the
sequence:
$$0 \longrightarrow \frac{R}{\ov{I^rJ^s}} \buildrel f\over\longrightarrow
\frac{R}{\ov{I^rJ^{s+l}}} \oplus \frac{R}{\ov{I^{r+k}J^s}}
\buildrel g\over\longrightarrow
\frac{(a^k,b^l)}{a^k\ov{I^{r}J^{s+l}}+b^l\ov{I^{r+k}J^s}} \longrightarrow 0$$
is exact, where $f(x^\prime)=(-(b^lx)^\prime,(a^kx)^\prime)$ and
$g(x^\prime,y^\prime)=(a^kx)^\prime+(b^ly)^\prime$. Here $'$ denotes image 
of an element in the respective ring. 
\end{lemma}
\begin{proof} It is easy to see that $\ker g=\image f$ and $g$ is surjective. We
prove that $f$ is injective. First we prove 
$(b^n) \cap \ov{I^rJ^{s+n}}=b^n \ov{I^rJ^s}$ for all $r,s \geq 0$ and $n \geq 1.$ Induct on $n.$
Since $(a,b)$ is a good joint reduction, the result is true for 
$n=1$. Let $n>1.$ Let $z=b^{n+1}y \in (b^{n+1}) \cap \ov{I^rJ^{s+n+1}} $. Then $
z \in (b^n) \cap \ov{I^rJ^{s+n+1}}=b^n \ov{I^r J^{s+1}}$ 
by induction hypothesis. 
Let $z=b^nu$ for some $u \in \ov{I^rJ^{s+1}}.$ Then $b^n(by-u)=0.$ Hence $by \in
(b) \cap \ov{I^rJ^{s+1}}=b\ov{I^rJ^s}.$ Thus 
$y \in \ov{I^rJ^s}$.  \\ 
Suppose $f(x^\prime)=0$. Then $b^lx \in (b^l) \cap \ov{I^rJ^{s+l}} =b^l
\ov{I^rJ^s}.$ Hence $x^\prime=0.$
\end{proof}

\begin{theorem} \label{thm2}
Fix integers $r,s \geq 0$. Then for $k,l \gg 0,$
\begin{eqnarray*}
\lm\left(\frac{\ov{I^{r+k}J^{s+l}}}{a^k\ov{I^rJ^{s+l}}+b^l\ov{I^{r+k}J^{s}}}
\right)&=&[\ov{e}_2(I)+\ov{e}_2(J)-\ov{e}_2(IJ)]+g_r(I,J)+
h_s(I,J)\\&+& rse(I|J)-\lm(R/\ov{I^rJ^s}).
\end{eqnarray*}
\end{theorem}
\begin{proof}
For all $k \geq 1$, 
$\displaystyle{\lm\left(\frac{\ov{I^{r+k}J^{s+l}}}{a^k\ov{I^rJ^{s+l}}+b^l\ov{I^{
r+k}J^{s}}}\right)}$

\begin{eqnarray*}
&=&\lm\left(\frac{R}{a^k\ov{I^rJ^{s+l}}+b^l\ov{I^{r+k}J^s}}\right)
-\lm(R/\ov{I^{r+k}J^{s+l}})\\
&=&\lm(R/(a^k,b^l))+\lm\left(\frac{(a^k,b^l)}{a^k\ov{I^rJ^{s+l}}+b^l\ov{I^{r+k}
J^s}}\right)
-\lm(R/\ov{I^{r+k}J^{s+l}})\\
&=& kl
e(I|J)+\lm(R/\ov{I^rJ^{s+l}})+\lm(R/\ov{I^{r+k}J^s})-\lm(R/\ov{I^rJ^s}
)-\lm(R/\ov{I^{r+k}J^{s+l}}),
\end{eqnarray*}
where the last equality follows from Lemma \ref{lemma2} and Rees's mixed
multiplicity theorem \cite[Theorem 2.4]{rees3}. Now 
\begin{eqnarray*}
 \lm(R/\ov{I^rJ^{s+l}})&=&\lm(R/\ov{J^{s+l}})+\lm(\ov{J^{s+l}}/\ov{I^rJ^{s+l}}).
\end{eqnarray*}
Therefore for $l \gg 0,$
\begin{eqnarray*}
 \lm(R/\ov{I^rJ^{s+l}})&=&
e(J)\binom{s+l+1}{2}+\ov{e}_2(J)+(e(I|J)r-\ov{e}_1(J))(s+l)+g_r(I,J)
\end{eqnarray*}
by Theorem \ref{thm1}(4). Similarly for $k \gg 0,$
\begin{eqnarray*}
 \lm(R/\ov{I^{r+k}J^{s}})&=&
e(I)\binom{r+k+1}{2}+\ov{e}_2(I)+(e(I|J)s-\ov{e}_1(I))(r+k)+h_s(I,J).
\end{eqnarray*}
Hence for $k,l \gg 0,$
\begin{eqnarray*}
\lm\left(\frac{\ov{I^{r+k}J^{s+l}}}{a^k\ov{I^rJ^{s+l}}+b^l\ov{I^{r+k}J^{s}}}
\right)&=&[\ov{e}_2(I)+\ov{e}_2(J)-\ov{e}_2(IJ)]+g_r(I,J)+
h_s(I,J)\\&+&rse(I|J)-\lm(R/\ov{I^rJ^s}).
\end{eqnarray*}
\end{proof}

\begin{theorem} \label{h} Let $(R,\m)$ be an analytically unramified
Cohen-Macaulay local ring of dimension $2$ with infinite residue 
field and $I$, $J$ be $\m$-primary ideals of $R$. Let $(a,b)$ be a good joint
reduction of $\mathcal I=\{\ov{I^rJ^s}\}$. Then for all 
$r,s \geq 0$,
\begin{enumerate}
 \item \noindent 
\begin{center}
$
[H^2_{(at_1,bt_2)}(\ov{\mathcal{R}^\prime})]_{(r,s)} \cong
\displaystyle\lim_{\stackrel{\longrightarrow}{k}}
\frac{\ov{I^{r+k}J^{s+k}}}{a^k\ov{I^rJ^{s+k}}+b^k\ov{I^{r+k}J^s}}.
$ 
                         \end{center}

\item For the directed system involved in the above direct limit, the map
$\mu_k$ is injective for $k\geq 1$, where $\mu_k$ is 
$$\frac{\ov{I^{r+k}J^{s+k}}}{a^k\ov{I^rJ^{s+k}}+b^k\ov{I^{r+k}J^s}}
\buildrel\mu_k\over\longrightarrow
\frac{\ov{I^{r+k+1}J^{s+k+1}}}{a^{k+1}\ov{I^rJ^{s+k+1}}+b^{k+1}\ov{I^{r+k+1}J^s}
}$$ the multiplication by $(ab)$. 
\item $\mu_k$ is surjective for all $r,s$ and large $k .$
\item 
 For $k \gg 0,$
$$ [H^2_{(at_1,bt_2)}(\ov{\mathcal{R}^\prime})]_{(r,s)}\cong
\frac{\ov{I^{r+k}J^{s+k}}}{a^k\ov{I^rJ^{s+k}}+b^k\ov{I^{r+k}J^s}}.$$

\end{enumerate}
\end{theorem}
\begin{proof} 
\begin{enumerate}
 \item 

Local cohomology modules have a natural
$\mathbb{Z}^2$-grading which is inherited from the $\mathbb{Z}^2$-grading of
$\ov{\mathcal{R}^\prime}$.
Therefore by equation \ref{eq1},
$$
[H^2_{(at_1,bt_2)}(\ov{\mathcal{R}^\prime})]_{(r,s)} =
\displaystyle\lim_{\stackrel{\longrightarrow}{k}}
\frac{\ov{I^{r+k}J^{s+k}}}{(\image \beta_k)_{(r,s)}}.
$$
Since $(\image \beta_k)_{(r,s)}
\simeq[a^k\ov{I^{r}J^{s+k}}+b^k\ov{I^{r+k}J^s}],$
$$
[H^2_{(at_1,bt_2)}(\ov{\mathcal{R}^\prime})]_{(r,s)} \cong
\displaystyle\lim_{\stackrel{\longrightarrow}{k}}
\frac{\ov{I^{r+k}J^{s+k}}}{a^k\ov{I^rJ^{s+k}}+b^k\ov{I^{r+k}J^s}}.
$$

\item
Let $x \in \ov{I^{r+k}J^{s+k}}$ such that $\mu_k(\bar{x})=0.$ Therefore 
$$xab \in a^{k+1} \ov{I^rJ^{s+k+1}}+b^{k+1}\ov{I^{r+k+1}J^s}.$$ Let
$xab=a^{k+1}p+b^{k+1}q,$ for some $p \in \ov{I^{r}J^{s+k+1}}$ and 
$q \in \ov{I^{r+k+1}J^{s}}.$ Since $(a,b)$ is a regular sequence, $p \in (b)
\cap \ov{I^{r}J^{s+k+1}}=b\ov{I^{r}J^{s+k}}$. 
Let $p=bp^\prime$ for some $p^\prime \in \ov{I^{r}J^{s+k}}.$ Similarly let
$q=aq^\prime$ for some $q^\prime \in \ov{I^{r+k}J^{s}}$. Thus 
$$xab=a^{k+1}bp^\prime+b^{k+1}aq^\prime.$$
Hence $x =a^kp^\prime+b^kq^\prime \in a^k\ov{I^rJ^{s+k}}+b^k\ov{I^{r+k}J^s}.$
Thus $\mu_k $ is injective.

\item From Theorem \ref{thm2},
$\lm\left(\frac{\ov{I^{r+k}J^{s+k}}}{a^k\ov{I^rJ^{s+k}}+b^k\ov{I^{r+k}J^s}}
\right)$ is independent 
of $k$ for $k \gg 0.$ Since $\mu_k$ is injective, $\mu_k$ is surjective for large $k$.
\item Follows from $(2)$ and $(3)$.
\end{enumerate}
\end{proof}

\noindent Following theorem is a consequence of Theorem \ref{thm2} and \ref{h}.

\begin{theorem} \label{thm3}
 Under the assumptions as above for all $r, s \geq 0,$ 
\begin{eqnarray*}
 \lm([H^2_{(at_1,bt_2)}(\ov{\mathcal{R}^\prime})]_{(r,s)})&=&[\ov{e}_2(I)+\ov{e}
_2(J)-\ov{e}_2(IJ)]+g_r(I,J)+h_s(I,J)\\&+&rse(I|J)-
\lm(R/\ov{I^rJ^s}).
\end{eqnarray*}
\end{theorem}

\noindent Next we analyse vanishing of
$\lm([H^2_{(at_1,bt_2)}(\ov{\mathcal{R}^\prime})]_{(r,s)})$.

\begin{theorem} \label{thm4}
Let $(R,\m)$ be an analytically unramified Cohen-Macaulay local ring of
dimension $2$ with infinite residue field. Let $(a,b)$ be a good joint reduction
of $\{\ov{I^rJ^s}\}$. 
Let $r_0, s_0 \geq 0$. Then following are equivalent:\\
\rm(1)
$\ov{e}_2(I)+\ov{e}_2(J)-\ov{e}_2(IJ)=\lm(R/\ov{I^{r_0}J^{s_0}})-g_{r_0}(I,J)-h_
{s_0}(I,J)-r_0s_0e(I|J),$\\
\rm(2) $[H^2_{(at_1,bt_2)}(\ov{\mathcal{R}^\prime})]_{(r_0,s_0)}=0,$\\
\rm(3) $\ov{I^{r+1}J^{s+1}}=a\ov{I^rJ^{s+1}}+b\ov{I^{r+1}J^s} \mbox{ for all }r
\geq r_0,~s\geq s_0.$
\end{theorem}
\begin{proof}
 Equivalence of $(1)$ and $(2)$ follows from Theorem \ref{thm3}. Now we prove
equivalence of $(2)$ and $(3)$.\\
$(2) \Longrightarrow (3)$: Let
$[H^2_{(at_1,bt_2)}(\ov{\mathcal{R}^\prime})]_{(r_0,s_0)}=0.$ Then by Theorem
\ref{h}(4) and \ref{thm2} 
for $r,s \gg 0$, say for $r,s \geq N,$ 
$$\ov{I^{r_0+r}J^{s_0+s}}=a^r\ov{I^{r_0}J^{s_0+s}}+b^s\ov{I^{r_0+r}J^{s_0}}.$$ 
We show the above equality for all $r,s \geq 1.$ First we show that for
$s\geq N,$
$$\ov{I^{r_0+N-1}J^{s_0+s}}=a^{N-1}\ov{I^{r_0}J^{s_0+s}}+b^s\ov{I^{r_0+N-1}J^{
s_0}}.$$ 
Let $x \in \ov{I^{r_0+N-1}J^{s+s_0}}$. Then $ax \in \ov{I^{r_0+N}J^{s_0+s}}. $
Let $ax=a^N u+b^s v$ for some $u \in \ov{I^{r_0}J^{s_0+s}}$ and 
$v \in \ov{I^{r_0+N}J^{s_0}}$. Thus $v \in (a) \cap
\ov{I^{r_0+N}J^{s_0}}=a\ov{I^{r_0+N-1}J^{s_0}}$. Let $v=av^\prime$ for some 
$v^\prime \in \ov{I^{r_0+N-1}J^{s_0}}$. Thus $x=a^{N-1}u+b^s v^\prime \in
a^{N-1}\ov{I^{r_0}J^{s_0+s}}+b^s\ov{I^{r_0+N-1}J^{s_0}}$. \\
Similar argument shows that for $r \geq N,$ 
$$\ov{I^{r_0+r}J^{s_0+N-1}}=a^r\ov{I^{r_0}J^{s_0+N-1}}+b^{N-1}\ov{I^{r_0+r}J^{
s_0}}.$$ 
Continuing as above we get for all $r,s \geq 1,$
$$\ov{I^{r_0+r}J^{s_0+s}}=a^r\ov{I^{r_0}J^{s_0+s}}+b^s\ov{I^{r_0+r}J^{s_0}} 
\subseteq a \ov{I^{r_0+r-1}J^{s_0+s}}+b\ov{I^{r_0+r}J^{s_0+s-1}}.$$
Hence for all $r\geq r_0,s \geq s_0$, $\ov{I^{r+1}J^{s+1}}=a\ov{I^rJ^{s+1}}+b\ov{I^{r+1}J^s}.$\\
$(3) \Longrightarrow (2)$:
Let \begin{equation} \label{eq2}
     \ov{I^{r+1}J^{s+1}}=a\ov{I^rJ^{s+1}}+b\ov{I^{r+1}J^s} \mbox{ for } r \geq
r_0, s \geq s_0.
    \end{equation}
 We prove 
$$\ov{I^{r+k}J^{s+k}}=a^k\ov{I^rJ^{s+k}}+b^k\ov{I^{r+k}J^s}, \mbox{ for } r\geq
r_0, s\geq s_0 \mbox{ and }k \geq 1. $$ 
Induct on $k.$ Equality is true for $k=1$. Let $k >1$. Then for $r \geq r_0,
s\geq s_0,$
\begin{eqnarray*}
\ov{I^{r+k}J^{s+k}}&=&a^{k-1}\ov{I^{r+1}J^{s+k}}+b^{k-1}\ov{I^{r+k}J^{s+1}},
\mbox{ by induction }\\ 
&=&a^{k-1}[a\ov{I^rJ^{s+k}}+b^{k-1}\ov{I^{r+1}J^{s+1}}]+b^{k-1}[a^{k-1}\ov{I^{
r+1}J^{s+1}}+b\ov{I^{r+k}J^s}],\\
&=&a^k\ov{I^rJ^{s+k}}+a^{k-1}b^{k-1}\ov{I^{r+1}J^{s+1}}+b^k\ov{I^{r+k}J^s}\\
&=&a^k\ov{I^rJ^{s+k}}+b^k\ov{I^{r+k}J^s}, \mbox{ by } (\ref{eq2}).
\end{eqnarray*}
Hence by Theorem \ref{h},
$[H^2_{(at_1,bt_2)}(\ov{\mathcal{R}^\prime})]_{(r_0,s_0)}=0$.
\end{proof}

\section{Application to Reduction Number and Rees's Theorem}
\noindent Let $(R,\m)$ be a Noetherian local ring of dimension $d$. 
Let $I$ be an $\m$-primary ideal. An ideal $K \subseteq I$ is said to be a
{\it reduction} of $\{\ov{I^n}\}$ if $K\ov{I^n}=\ov{I^{n+1}}$ 
for all large $n$. A {\it minimal reduction} of $\{\ov{I^n}\}$ is a reduction of
$\{\ov{I^n}\}$ minimal with respect to inclusion. For a 
minimal reduction $K$ of $\{\ov{I^n}\}$, we set $$\ov{r}_K(I)= \sup \{n\in
\mathbb Z \mid \ov{I^n}\not= K\ov{I^{n-1}}\}.$$ 
The {\it reduction  number} $\ov{r}(I)$ of $\{\ov{I^n}\}$ is defined to be the
least $\ov{r}_K(I)$ over all possible minimal reductions 
$K$ of $\{\ov{I^n}\}$. In this section we derive Rees's Theorem \cite[Theorem
2.5]{rees} as a consequence of Theorem \ref{thm4}. 
We also derive a theorem due to T. Marley which asserts that $\ov{r}(I) \leq
k+1$ for some integer $k \geq 0$ if and only if $\lm(R/\ov{I^k})=\ov{P}_I(k)$ 
\cite[Corollary 3.8 and 3.12]{marleythesis}. 

\begin{theorem} \cite[Theorem 2.5]{rees} \label{thm7}
 Let $(R,\m)$ be an analytically unramified Cohen-Macaulay local ring of
dimension $2$ with infinite residue field. Let $I,J$ be $\m$-primary 
ideals and $(a,b)$ be a good joint reduction of $\{\ov{I^rJ^s}\}$. Then
the following are equivalent:\\
\rm(1) $\ov{e}_2(IJ)=\ov{e}_2(I)+\ov{e}_2(J),$\\
\rm(2) $[H^2_{(at_1,bt_2)}(\ov{\mathcal{R}^\prime})]_{(0,0)}=0,$\\
\rm(3) for all $r,s >0,$
$$\ov{I^rJ^s}=a\ov{I^{r-1}J^s}+b\ov{I^rJ^{s-1}}.$$
\end{theorem}
\begin{proof}
Let $r_0=s_0=0$ in Theorem \ref{thm4}. Then $g_{r_0}(I,J)=h_{s_0}(I,J)=0$.
Therefore by Theorem \ref{thm4}, the result follows.
\end{proof}

\begin{theorem} \cite[Corollary 3.8]{marleythesis}
 Let $(R,\m)$ be an analytically unramified Cohen-Macaulay local ring of
dimension $2$ with infinite residue field. Then $\ov{r}(I) \leq 
k+1$ for some integer $k \geq 0$ if and only if $\lm(R/\ov{I^{k}})=\ov{P}_I(k)$.
\end{theorem}
\begin{proof}
By Theorem \ref{thm1} for any $k \geq 0$ and $n \gg0$,
$$\lm(\ov{I^n}/\ov{I^{n+k}})=kn e(I)+g_k(I,I).$$ 
For $n \gg 0$, \begin{eqnarray*}
      \lm(\ov{I^n}/\ov{I^{n+k}})&=&\lm(R/\ov{I^{n+k}})-\lm(R/\ov{I^n})\\
&=&
e(I)\binom{n+k+1}{2}-\ov{e}_1(I)(n+k)+\ov{e}_2(I)\\&-&[e(I)\binom{n+1}{2}-\ov{e}
_1(I)n+\ov{e}_2(I)] \\
&=& e(I)kn+e(I)\binom{k+1}{2}-\ov{e}_1(I)k.
     \end{eqnarray*}
Hence $g_k(I,I)=e(I)\binom{k+1}{2}-\ov{e}_1(I)k$. Put $s_0=0$ and $r_0=k$ in
Theorem \ref{thm4} and use $\ov{e}_2(I)=\ov{e}_2(I^2)$ to get 
$$\ov{e}_2(I)+e(I)\binom{k+1}{2}-\ov{e}_1(I)k -\lm(R/\ov{I^k})=0 \mbox{ if and
only if } \ov{I^{n+2}}=(a,b)\ov{I^{n+1}}$$ 
for $n \geq k$, where $(a,b)$ is reduction 
of $\{\ov{I^n}\}$ such that $a^*,b^*\in \ov{I}/\ov{I^2}$ are nonzerodivisors in
$\bigoplus_{n \geq 0}\ov{I^n}/\ov{I^{n+1}}$. Since 
$\ov{r}(I)$ is independent of the minimal reduction chosen \cite[Corollary
3.8]{marleythesis}, $\ov{r}(I) \leq k+1$ if and only if
$\lm(R/\ov{I^k})=\ov{P}_I(k)$.
\end{proof}

\section{Vanishing of $\ov{e}_{2}(IJ)$}
\noindent In this section we analyze vanishing of $\ov{e}_{2}(IJ)$ in an analytically
unramified Cohen-Macaulay local ring of dimension $2$ with 
infinite residue field. We prove that $\ov{e}_{2}(IJ)=0$ if and only if
$\ov{\mathcal{R}}(I,J)$ is Cohen-Macaulay. 
\begin{definition} We 
say that the {\it normal joint reduction number} of $I$ and $J$ is zero with respect
to a joint reduction $(a,b)$ of $\{\ov{I^rJ^s}\}$ 
if for all $r,s \geq 1,$
$$\ov{I^rJ^s}=a\ov{I^{r-1}J^s}+b\ov{I^rJ^{s-1}}.$$
\end{definition}
\noindent We need an analogue of Grothendieck-Serre difference formula for
$\{\ov{I^rJ^s}\}$ \cite[Theorem 5.1]{jayanthan-verma}. 
The proof of this is similar to \cite[Theorem 5.1]{jayanthan-verma}. 

\begin{theorem}\rm[The Difference Formula] \label{difference formula}
Let $(R,\m) $ be an analytically unramified local ring of dimension $d$. 
Then for all $r,s \geq 0$,
\begin{enumerate}
 \item
$\lambda_R([H^i_{\mathcal{R}^\prime_{++}}(\ov{\mathcal{R}^\prime})]_{(r,s)}) <
\infty$.
\item
$\ov{P}_{I,J}(r,s)-\ov{H}_{I,J}(r,s)=\sum_{i=0}^d(-1)^i\lambda_R([H^i_{\mathcal{
R}^\prime_{++}}(\ov{\mathcal{R}^\prime})]_{(r,s)})$.
\end{enumerate}
\end{theorem}

\noindent M. Herrmann, E. Hyry, J. Ribbe and Z. Tang proved that if $R$ 
is Cohen-Macaulay local ring and joint reduction number of $I$ and $J$ is zero then 
Cohen-Macaulayness of $\mathcal R(I,J)$ is equivalent to Cohen-Macaulayness of 
$\mathcal R(I)$ and $\mathcal R(J)$ \cite[Corollary 3.4]{hhrz}. The same proof 
works for the filtration $\{\ov{I^rJ^s}\}$. For the sake of completeness we give 
a proof. 
\begin{theorem} \label{thm10}
 Let $(R,\m)$ be an analytically unramified Cohen-Macaulay local ring of
dimension $2$ and $I,J$ be $\m$-primary ideals. Let the normal joint 
reduction number of $I$ and $J$ be zero with respect to some good joint
reduction. Then $\ov{\mathcal R}(I,J)$ is
Cohen-Macaulay if and 
only if $\ov{\mathcal R}(I)=\bigoplus_{n \geq 0}\ov{I^n}t^n$ and $\ov{\mathcal
R}(J)=\bigoplus_{n \geq 0}\ov{J^n}t^n$ are Cohen-Macaulay.
\end{theorem}
\begin{proof}
 Let $(a,b)$ be a good joint reduction of $\{\ov{I^rJ^s}\}$ such that 
$$\ov{I^rJ^s}=a\ov{I^{r-1}J^s}+b\ov{I^rJ^{s-1}}$$ for all $r,s \geq 1.$ Let 
$$S:=\ov{\mathcal R}(I,J)/(at_1,bt_2)=R \oplus S_1^+ \oplus S_2^+,$$
where $S_1^+=\bigoplus_{n \geq 1} \ov{I^n}/a\ov{I^{n-1}}$ and
$S_2^+=\bigoplus_{n \geq 1}\ov{J^n}/b\ov{J^{n-1}}$. Note that 
\begin{eqnarray*}
 S_1&=&R \oplus S_1^+ \simeq \ov{\mathcal{R}}(I)/(at)\\
\mbox{ and }S_2 &=& R \oplus S_2^+ \simeq \ov{\mathcal R}(J)/(bt).
\end{eqnarray*}
The elements $at$ and $bt$ are nonzerodivisors in $\ov{\mathcal R}(I)$ and
$\ov{\mathcal R}(J)$ respectively. Hence 
$\ov{\mathcal R}(I)$ and $\ov{\mathcal R}(J)$ are Cohen-Macaulay if and only if
$S_1$ and $S_2 $ are Cohen-Macaulay. By Lemma \cite[Lemma 3.2]
{hhrz} it is enough to prove that $\ov{\mathcal R}(I,J)$ is Cohen-Macaulay if
and only if $S$ is Cohen-Macaulay. 
Since $S=\ov{\mathcal{R}}(I,J)/(at_1,bt_2)$ we show that $(at_1,bt_2)$ is a
regular sequence on 
$\ov{\mathcal R}(I,J).$ We prove that $bt_2$ is $\ov{\mathcal
R}(I,J)/(at_1)$-regular. Let $z bt_2\in(at_1).$ We may 
assume $z=vt_1^rt_2^s$ for some $v \in \ov{I^rJ^s}$ and $r>0$. Then $vb =a w$
for some $w \in \ov{I^{r-1}J^{s+1}}.$ Thus $v \in (a) \cap \ov{I^rJ^s} =
a\ov{I^{r-1}J^s}.$ Hence $z \in (at_1). $ Therefore $(at_1,bt_2) $ is a regular
sequence.  
\end{proof}

\noindent We now prove a few preliminary results in order to show that 
Cohen-Macaulayness of $\ov{\mathcal R}(I,J)$ implies that of $\ov{\mathcal R}(IJ)$ 
if $R$ is analytically unramified local ring and $I,J$ have positive height. We recall 
some definitions and notation 
from \cite{hyry}. Let $T$ be an $\mathbb{N}^2$-graded ring defined over a local ring and let $M$ be a 
finitely generated $\mathbb{N}^2$-graded $T$-module. Let $\mathcal{M}$ be the maximal homogeneous 
ideal of $T$. E. Hyry \cite{hyry} defined the {\it $a$-invariants} of $T$ as:
\begin{eqnarray*}
 a^1(M)&=&\sup\{k\in \mathbb{Z}\mid [H^{\dim M}_{\mathcal{M}}(M)]_{(k,q)} \neq 0 \mbox{ for some }q \in \mathbb{Z}\}\\
a^2(M)&=&\sup\{k\in \mathbb{Z}\mid [H^{\dim M}_{\mathcal{M}}(M)]_{(p,k)} \neq 0 \mbox{ for some }p \in \mathbb{Z}\}.
\end{eqnarray*}
Let $M^{\Delta}=\bigoplus_{n}M_{(n,n)}, T_1=\bigoplus_{p\geq 0}T_{(p,0)}, T_2=\bigoplus_{q \geq 0}T_{(0,q)}$ 
and $T^+=\bigoplus_{p,q \geq 1}T_{(p,q)}$. E. Hyry proved that if $T$ is standard 
bigraded Cohen-Macaulay ring of dimension $d+2$ such that 
$\dim T_1,\dim T_2<d+2$ and $a^1(T),a^2(T)<0$ then $T^{\Delta}$ is 
Cohen-Macaulay \cite[Theorem 2.5]{hyry}. The same proof shows that 
if $M$ is Cohen-Macaulay of dimension $d+2$ such that 
$a^1(M),a^2(M)<0$ then $M^\Delta $ is Cohen-Macaulay. For 
the sake of completeness we give a proof. We recall following change 
of grading principle from \cite{hyry}.

\begin{theorem} \cite{hyry}
 Let $T$ be $\mathbb{Z}^r-$graded ring and $\mathcal{U}$ be a homogeneous ideal 
of $T$. Given a homomorphism $\phi: \mathbb{Z}^r 
\rightarrow \mathbb{Z}^q$, set 
$$M^\phi=\bigoplus_{{\bf m}\in \mathbb{Z}^q}\left( \bigoplus_{\phi({\bf n})={\bf m}}M_{\bf n}\right)$$
for any $r-$graded $T-$module $M$ . Then 
$$(H^i_\mathcal U(M))^\phi=H^i_{\mathcal U^\phi}(M^\phi).$$
\end{theorem}

\begin{theorem} \label{thm12}
 Let $T$ be a standard bigraded ring and let $M$ be a $\mathbb{N}^2-$graded Cohen-Macaulay 
$T-$module of dimension $d+2$ such that 
$\dim M^\Delta=d+1$ and $a^1(M),a^2(M)<0$. Then $M^\Delta$ is Cohen-Macaulay.
\end{theorem}
\begin{proof}
Set $\mathcal M_1^+=(\m \bigoplus T_{1}^+)T,\mathcal M_2^+=(\m \bigoplus T_{2}^+)T, 
\mathcal{M}^+=(\m\bigoplus T^+)T$. First we prove that if $i<d+1$ 
then $[H^i_{\mathcal M^+}(M)]_{(p,q)}=0$ for $p,q \geq 0$ or $p,q<0$.
We have $\mathcal{M}_1^+ + \mathcal{M}_2^+=\mathcal{M}$ and $\mathcal{M}_1^+
 \cap \mathcal{M}_2^+=\mathcal{M}^+$. Consider the 
Mayer-Vietoris sequence of local cohomology modules
\begin{eqnarray*}
 \cdots \rightarrow H^i_\mathcal M(M) \rightarrow H^i_{\mathcal M_1^+}(M) \bigoplus H^i_{\mathcal M_2^+}(M) 
\rightarrow H^i_{\mathcal{M}^+}(M) \rightarrow \cdots .
\end{eqnarray*}
Since $H^i_{\mathcal M}(M)=0$ for $i<d+2$, the homomorphism
\begin{eqnarray}\label{eq9}
 H^i_{\mathcal M_1^+}(M)\bigoplus H^i_{\mathcal M_2^+}(M) \rightarrow H^i_{\mathcal M^+}(M)
\end{eqnarray}
is an isomorphism for $i<d+1$. 
Let $M(-,q)=\bigoplus_{p\geq 0} M_{(p,q)}$ 
and $M(p,-)=\bigoplus_{q\geq 0}M_{(p,q)}$. Since $M=\bigoplus_{q \geq 0}M(-,q)=
\bigoplus_{p\geq 0}M(p,-)$ by change of grading principle, we have
$$[H^i_{\mathcal{M}_1^+}(M)]_{(p,q)}=[H^i_{{\mathcal M}(-,0)}(M(-,q))]_p \mbox{ and }
[H^i_{\mathcal M_2^+}(M)]_{(p,q)}=[H^i_{{\mathcal M}(0,-)}(M(p,-))]_q$$
for all $p,q \in \mathbb{Z}$ and $i \geq 0$. See \cite{hyry}. Therefore 
$$[H^i_{\mathcal M_1^+}(M)]_{(p,q)}=0 \mbox{ if }q<0 \mbox{ and }
[H^i_{\mathcal{M}_2^+}(M)]_{(p,q)}=0 \mbox{ if }p<0.$$ 
Hence by \ref{eq9}, $[H^i_{\mathcal M^+}(M)]_{(p,q)}=0$ for $p,q<0$ and $i<d+1$. 
Let $\mathcal{N}$ be the maximal homogeneous ideal of $T_1$. Then $M_\mathcal N=
\bigoplus_{q \geq 0}M(-,q)_{\mathcal{N}}$ is 
$\mathbb{N}$-graded module over a local ring $T_1{_\mathcal N}$. Since 
$[H^i_{\mathcal M}(M)]_{(p,q)}=0$ for all $q \geq 0$ and $i\geq 0$, 
$[H^i_{\mathcal P}(M_\mathcal N)]_q=0$ for all $q \geq 0$ and $i \geq 0$, 
where $\mathcal P=\mathcal N(T_1){_\mathcal N} \bigoplus_{q>0} [T_\mathcal N]_q$ 
is maximal homogeneous ideal of $T_\mathcal{N}$. 
Therefore by Lemma \cite[Lemma 2.3]{hyry}, 
$[H^i_{\mathcal{M}_2^+}(M)]_{(p,q)}=0$ for all $q \geq 0 $ and $i \geq 0$. Similarly 
$[H^i_{\mathcal{M}_1^+}(M)]_{(p,q)}=0$ for all $p\geq 0$ and $i\geq 0$. Hence 
by \ref{eq9}, $[H^i_{\mathcal M^+}(M)]_{(p,q)}=0$ for all $p,q \geq 0$ and $i<d+1$.
Therefore by Lemma \cite[Lemma 2.2]{hyry}, 
\begin{eqnarray*}
 [H^i_{\mathcal{M}^\Delta}(M(p,0))^\Delta]_q=0 \mbox{ if }p\geq 0, q \notin\{-p,\ldots,-1\} \mbox{ and }i<d+1. 
 \end{eqnarray*}
Hence $H^i_{\mathcal M^\Delta}(M^\Delta)=0$ for $i<d+1$. Therefore $M^\Delta$ is Cohen-Macaulay.
\end{proof}

\noindent Let $\mathcal{R}=\mathcal R(I,J)$ and 
$\ov{\mathcal R}=\ov{\mathcal{R}}(I,J)$. Suppose $S$ is a graded ring defined over a local ring 
and $\mathcal{M}$ is the maximal homogeneous ideal of $S$. For $\mathbb{N}$-graded module $M$
let $$a(M):=\max\{k\mid [H^{\dim M}_\mathcal M(M)]_k \neq0\}.$$ M. Herrmann, E. Hyry and J. Ribbe 
proved that for homogeneous ideal $I$ of positive height in a multi-graded ring $B$, 
$a(\mathcal{R}(I))=-1$ \cite{hhr}. The same proof shows the 
following:

\begin{lemma} \label{lemma3}\rm \cite[Lemma 2.1]{hhr}
 Let $B$ be a multi-graded ring of dimension $d$ defined over a local ring and 
let $I \subseteq B$ be a homogeneous ideal of positive height. Then $a(\ov{\mathcal R}(I))=-1$.
\end{lemma}
\begin{proof}
By localizing at maximal homogeneous ideal of $B$ we may assume that $B$ 
is local. Let $\ov{\mathcal R}(I)_+=\bigoplus_{n>0}\ov{I^n}t^n$ and 
$ \ov{G}=\bigoplus_{n\geq 0}\ov{I^n}/\ov{I^{n+1}}$. Consider the exact sequences 
$$ 0 \longrightarrow \ov{\mathcal{R}}(I)_+\longrightarrow \ov{\mathcal{R}}(I) \longrightarrow B \longrightarrow 0$$
$$0 \longrightarrow \ov{\mathcal{R}}(I)_+(1) \longrightarrow \ov{\mathcal{R}}(I) \longrightarrow \ov{G} \longrightarrow 0.$$
Let $\mathcal M$ be the maximal homogeneous ideal of $\mathcal R(I)$. 
Therefore we get the long exact sequence of local cohomology modules
$$\cdots \longrightarrow H^i_\mathcal M(\ov{\mathcal R}(I)_+) \longrightarrow H^i_\mathcal M(\ov{\mathcal R}(I)) 
\longrightarrow H^i_\mathcal{M}(B) \longrightarrow \cdots \mbox{ and }$$
$$\cdots \longrightarrow H^i_\mathcal M(\ov{\mathcal R}(I)_+(1)) \longrightarrow H^i_\mathcal M(\ov{\mathcal R}(I)) 
\longrightarrow H^i_\mathcal{M}(G)\longrightarrow \cdots .$$
This gives the isomorphisms 
$$[H^{d+1}_\mathcal M(\ov{\mathcal R}(I)_+)]_n \longrightarrow [H^{d+1}_\mathcal M(\ov{\mathcal R}(I))]_n \mbox{ for }n \neq 0$$
and the epimorphisms
$$[H^{d+1}_\mathcal M(\ov{\mathcal R}(I)_+)]_{n+1} \longrightarrow [H^{d+1}_\mathcal M(\ov{\mathcal R}(I))]_n \mbox{ for }n \in \mathbb Z.$$
Since $[H^{d+1}_\mathcal M(\ov{\mathcal R}(I)_+)]_n=0$ for $n\gg0$, $[H^{d+1}_\mathcal{M}(\ov{\mathcal R}(I))]_n=0$ 
for $n\geq 0$. If $[H^{d+1}_\mathcal M(\ov{\mathcal R }(I))]_{-1}=0$ then 
$H^{d+1}_\mathcal M(\ov{\mathcal R}(I))=0,$ a contradiction. Thus 
$a(\ov{\mathcal R}(I))=-1$.
\end{proof}

\begin{theorem} \label{thm13}
 Let $R$ be an analytically unramified local ring and let $I,J$ be ideals of positive height. 
If $\ov{\mathcal R}(I,J)$ is 
Cohen-Macaulay then $\ov{\mathcal R}(IJ)$ is Cohen-Macaulay.
\end{theorem}
\begin{proof}
 Let $B=\ov{\mathcal R}(I)=\bigoplus_{n \geq 0}\ov{I^n}t_1^n$ and $K=\bigoplus_{n \geq 0}JI^nt_1^n$ be homogeneous 
ideal in $B$. First we prove $\ov{K^s}=\bigoplus_{n \geq 0} \ov{J^sI^n}t_1^n$. Let $zt_1^k \in 
R[t_1]$ be integral over $K^s$. Then 
\begin{eqnarray*}
 (zt_1^k)^n+b_1(zt_1^k)^{n-1}+\cdots+b_n=0
\end{eqnarray*}
 for some $b_i\in (K^s)^i$. Comparing coefficient of $t_1^{nk}$, we get 
$$z^n+c_1z^{n-1}+\cdots+c_n=0$$
for some $c_i\in (J^sI^k)^i$. Hence $z \in \ov{J^sI^k}$. Therefore 
$\ov{K^s}=\bigoplus_{n \geq 0} \ov{J^sI^n}t_1^n$ and $\ov{\mathcal R}_B(K)=\ov{\mathcal R}(I,J)$. 
Therefore by Lemma \ref{lemma3}, 
$a^2(\ov{\mathcal R}(I,J))=-1$. Similar argument shows that $a^1(\ov{\mathcal R}(I,J))=-1$. 
Therefore by Theorem \ref{thm12}, result follows.
\end{proof}

\noindent As a consequence of Theorem \ref{thm13} we get sufficient condition for 
vanishing of $\ov{e}_d(IJ)$.

\begin{corollary} \label{corollary} \label{corollary1}
 Let $(R,\m)$ be an analytically unramified Cohen-Macaulay local ring of 
dimension $d$. If $\ov{R}(I,J)$ 
is Cohen-Macaulay then $\ov{e}_d(IJ)=0$. 
\end{corollary}
\begin{proof}
 Since $\ov{R}(I,J)$ is Cohen-Macaulay by Theorem \ref{thm13}, $\ov{R}(IJ)$ is Cohen-Macaulay. 
Therefore by \cite[Theorem 2.3]{viet}, $\ov{r}(IJ) \leq d-1$. Hence by \cite[Corollary 3.17]{marleythesis}, 
$\ov{e}_d(IJ)=0$.
\end{proof}

\noindent Converse of Theorem \ref{corollary1} is true in $2-$dimensional 
analytically unramified Cohen-Macaulay local ring. In order to prove 
this we need the theorem due to D. Rees which asserts that 
$\ov{e}_2(IJ) \geq \max\{\ov{e}_2(I),\ov{e}_2(J)\}$. 
We give a cohomological proof. 

\begin{lemma} \label{lcws++}
 Let $(R,\m)$ be an analytically unramified Cohen-Macaulay local ring 
of dimension $2$ with infinite residue field and $u=u_1u_2$. Then for all $r,s \geq -1$
$$[H^0_{\ov{\mathcal R^\prime}_{++}}(\ov{\mathcal R^\prime}/u\ov{\mathcal R^\prime})]_{(r,s)}=0.$$
\end{lemma}
\begin{proof}
 Let $(a,b)$ be a good joint reduction of $\{\ov{I^rJ^s}\}$. Then for all $r,s\geq 0,$
$$(ab) \cap \ov{I^{r+1}J^{s+1}}=ab \ov{I^rJ^s}.$$
Let $(ab)^*$ denote the image of $abt_1t_2$ in $[\ov{\mathcal R^\prime}/u\ov{\mathcal R^\prime}]_{(1,1)}$. 
Then $[(0:_{\ov{\mathcal R^\prime}/u\ov{\mathcal R^\prime}} (ab)^*)]_{(r,s)}=0$ 
for all $r,s \geq -1$. For fixed $r,s \geq -1,$ let $z^* \in [H^0_{\ov{\mathcal R^\prime}_{++}}
(\ov{\mathcal R^\prime}/u\ov{\mathcal R^\prime})]_{(r,s)}$ for some $z \in \ov{I^rJ^s}$. 
Then $(ab)^{*n}z^* =0$ for some $n$. Therefore $z^* =0$.  
\end{proof}

\begin{theorem} \label{thm6}
 Let $(R,\m)$ be an analytically unramified Cohen-Macaulay local ring of dimension $d$ 
with infinite residue field. Put
$M=(u,\mathcal{R}^\prime_{++})\mathcal{R}^\prime$ and
$N=\mathcal{R}^\prime_{++}$. 
Then \\
\noindent \rm(1) there exist isomorphisms 
$$H^i_M(\ov{\mathcal{R}^\prime}) \longrightarrow H^i_N(\ov{\mathcal{R}^\prime})
\mbox{ for } i=0,1,\ldots,d-1$$
and an exact sequence 
\begin{eqnarray*}
 0 \longrightarrow H^d_M(\ov{\mathcal{R}^\prime}) \longrightarrow
H^d_N(\ov{\mathcal{R}^\prime}) \longrightarrow H^d_\m(R)[t_1,t_2,u_1,u_2] 
\longrightarrow H^{d+1}_M(\ov{\mathcal{R}^\prime}) \longrightarrow 0.
\end{eqnarray*}
\rm(2) For $d= 2,$
$H^0_{N}(\ov{\mathcal{R}^\prime})=0$ and $[H^1_N(\ov{\mathcal{R}^\prime})]_{(r,s)}=0$ 
for all $r,s \geq 0$.
\end{theorem}
\begin{proof}
 \rm(1) Since $\ov{\mathcal R^\prime}_{u}=R[t_1,t_2,u_1,u_2],$ we have 
$$H^i_\m(R)[t_1,t_2,u_1,u_2] \cong H^i_\m(R[t_1,t_2,u_1,u_2]) \cong
H^i_N(\ov{\mathcal{R}^\prime}_u).$$
See \cite[Theorem 4.3.2]{bs}. Therefore we get a long exact sequence of local
cohomology modules 
\begin{eqnarray*}
  \cdots \rightarrow H^i_M(\ov{\mathcal R^\prime}) \rightarrow
H^i_N(\ov{\mathcal R^\prime}) \rightarrow
 H^i_\m(R)[t_1,t_2,u_1,u_2] \rightarrow H^{i+1}_M(\ov{\mathcal{R}^\prime})
\rightarrow \cdots . 
\end{eqnarray*}
See \cite[Exercise 5.1.22]{bs}.
Since $R$ is Cohen-Macaulay, $H^i_\m(R)=0$ for $i=0,\ldots,d-1. $ This gives the
desired result.\\
\rm(2) Since $N\ov{\mathcal{R}^\prime}$ and 
$\ov{\mathcal{R}^\prime}_{++}$ have the same radical,
$$H^i_{N}(\ov{\mathcal{R}^\prime})\cong
H^i_{N\ov{\mathcal{R}^\prime}}(\ov{\mathcal{R}^\prime}) \cong
H^i_{\ov{\mathcal{R}^\prime}_{++}}
(\ov{\mathcal{R}^\prime}).$$ 
Consider the exact sequence 
\begin{eqnarray*}
 \diagram
0 \rto& \ov{\mathcal{R}^\prime} \rto^{u}& \ov{\mathcal{R}^\prime} \rto& \ov{\mathcal{R}^\prime}/u\ov{\mathcal R^\prime} \rto &0. 
\enddiagram
\end{eqnarray*}
This gives a long exact sequence 
$$
\cdots \longrightarrow [H^i_M(\ov{\mathcal R^\prime})]_{(r+1,s+1)} \longrightarrow [H^i_M(\ov{\mathcal R^\prime})]_{(r,s)} 
\longrightarrow [H^i_M(\ov{\mathcal R^\prime}/u\ov{\mathcal R^\prime})]_{(r,s)} \longrightarrow \cdots.
$$
By Lemma \ref{lcws++}, for $r,s \geq -1,$ the map 
$$\mu_u: [H^1_M(\ov{\mathcal R^\prime})]_{(r+1,s+1)} \longrightarrow [H^1_M(\ov{\mathcal R^\prime})_{(r,s)}$$
is injective. Hence $[H^1_M(\ov{\mathcal R^\prime})]_{(r,s)}=0$ for all $r,s \geq 0$. 
Since $H^i_M(\ov{\mathcal R^\prime})=H^i_N(\ov{\mathcal R^\prime})$ for $i=1,2$ 
result follows. 
\end{proof}

\begin{theorem} \cite[Theorem 2.4(iii)]{rees} \label{thm8}
Let $(R,\m)$ be an analytically unramified Cohen-Macaulay local ring of
dimension $2.$ Then 
$$\ov{e}_2(IJ)\geq \max\{\ov{e}_2(I),\ov{e}_2(J)\}.$$ 
\end{theorem}
\begin{proof} 
Using Theorem \ref{difference formula} and \ref{thm6},
\begin{eqnarray*}
 \ov{P}_{(I,J)}(r,0)-\ov{H}_{(I,J)}(r,0)=\lm([H^2_{\mathcal{R}^\prime_{++}}(\ov{
\mathcal{R}^\prime})]_{(r,0)}) \mbox{ for all }r \geq 0. 
\end{eqnarray*}
Hence
\begin{eqnarray*}
 \ov{P}_I(r)-\ov{H}_I(r)+[\ov{e}_2(IJ)-\ov{e}_2(I)] &=&
\lm([H^2_{\mathcal{R}^\prime_{++}}(\ov{\mathcal{R}^\prime})]_{(r,0)}) \geq 0.
\end{eqnarray*}
Taking $r \gg0$ we get $\ov{e}_2(IJ) \geq \ov{e}_2(I).$ Similarly $\ov{e}_2(IJ)
\geq \ov{e}_2(J).$ Hence 
$\ov{e}_2(IJ) \geq \max\{\ov{e}_2(I),\ov{e}_2(J)\}.$
 \end{proof}
 
\noindent Following theorem gives necessary and sufficient condition 
for vanishing of $\ov{e}_2(IJ)$. E. Hyry proved 
equivalence of $(2)$ and $(3)$ for $(I,J)-$adic 
case \cite[Corollary 3.5]{hyry}. 
\begin{theorem} \label{thm14}
 Let $(R,\m)$ be an analytically unramified Cohen-Macaulay local ring of
dimension $2$ with infinite residue field. Then 
following statements are equivalent:
\begin{enumerate}
 \item $\ov{e}_{2}(IJ)=0,$
\item the normal joint reduction number of $I$ and $J$ is zero with respect to any
good joint reduction and $\ov{r}(I),\ov{r}(J) \leq 1,$ 
\item $\ov{\mathcal{R}}(I,J)$ is Cohen-Macaulay.
\end{enumerate} 
In this case $\ov{P}_{I,J}(r,s)=\lm(R/\ov{I^rJ^s})$ for all $r,s \geq 0$ and 
$$\ov{P}_{I,J}(x,y)=\ov{P}_I(x)+e(I|J)xy+\ov{P}_J(y).$$
\end{theorem}
\begin{proof}
$(1) \Longrightarrow (2)$: Let $\ov{e}_{2}(IJ)=0.$ Since $\ov{e}_2(I) \geq 0$ in
an analytically unramified Cohen-Macaulay local
ring \cite[Proposition 3.23]{marleythesis}, $\ov{e}_2(I)=\ov{e}_2(J)=0$ by
Theorem \ref{thm8}. Therefore $\ov{r}(I),\ov{r}(J) \leq 1$ 
by \cite[Corollary 3.26]{marleythesis}. 
Hence the normal joint reduction number of $I$ and $J$ is zero
with respect to any good joint reduction $(a,b)$, by Theorem \ref{thm7}. \\
$(2) \Longrightarrow (3):$ Let the normal joint reduction number of $I$ and $J$ be
zero with respect to any good joint reduction and 
$\ov{r}(I),\ov{r}(J) \leq 1.$ Therefore $\ov{e}_1(I)=\lm(\ov{I}/K)$ for any minimal 
reduction $K$ of $\{\ov{I^n}\}$ and $\bigoplus_{n \geq 0}\ov{I^n}/\ov{I^{n+1}}$ is Cohen-Macaulay. 
See \cite[Theorem 4.7]{huckaba-marley}. Thus $\ov{\mathcal{R}}(I)$ is
Cohen-Macaulay by \cite[Theorem 2.3]{viet}. Similarly $\ov{\mathcal{R}}(J)$ 
is Cohen-Macaulay. Therefore by Theorem \ref{thm10}, $\ov{\mathcal R}(I,J)$ is Cohen-Macaulay.\\
$(3) \Longrightarrow (1):$ Follows from Corollary \ref{corollary}.

In order to establish the formula for $\ov{P}_{I,J}(x,y)$, let $(a,b)$ be a good joint reduction. 
Since the normal joint reduction number is
zero with respect to any good joint reduction, for all 
$r,s \geq 1$ we have
$$\ov{I^rJ^s}=a\ov{I^{r-1}J^s}+b\ov{I^rJ^{s-1}}.$$
Using induction on $r,s$ we get 
\begin{eqnarray*}\label{eq7}
\ov{I^rJ^s}=a^r\ov{J^s}+b^s\ov{I^r}
\end{eqnarray*}
for all $r,s \geq 1$. 
Therefore for all $r,s \geq 1$,
\begin{eqnarray*}
 \lm(R/\ov{I^rJ^s})&=& \lm(R/a^r\ov{J^s}+b^s\ov{I^r})\\
&=&\lm(R/(a^r,b^s))+\lm((a^r,b^s)/a^r\ov{J^s}+b^s\ov{I^r})\\
&=&rs\lm(R/(a,b))+\lm(R/\ov{I^r})+\lm(R/\ov{J^s}) \mbox{ by Lemma }\ref{lemma2}.
\end{eqnarray*} 
Since $\lm(R/(a,b))=e(I|J),$ by \cite[Theorem 17.4.9]{huneke-swanson}, 
\begin{eqnarray}\label{eq6}
\lm(R/\ov{I^rJ^s})=rse(I|J)+\lm(R/\ov{I^r})+\lm(R/\ov{J^s}).
\end{eqnarray}
If $r=0$ or $s=0$ the equation 
$$\lm(R/\ov{I^rJ^s})=rse(I|J)+\lm(R/\ov{I^r})+\lm(R/\ov{J^s})$$
is still true. Since $\ov{e}_2(I)=0$, $\ov{P}_I(r)=\lm(R/\ov{I^r})$ for all $r
\geq 0$ \cite[Corollary 3.8]{marleythesis}. Similarly $\ov{P}_J(s)=
\lm(R/\ov{J^s})$ for all $s \geq 0$. Therefore for all $r,s \geq 0$ 
$$\lm(R/\ov{I^rJ^s})=rse(I|J)+\ov{P}_I(r)+\ov{P}_J(s).$$ Hence
$$\ov{P}_{I,J}(x,y)=e(I|J)xy+\ov{P}_I(x)+\ov{P}_J(y).$$
\end{proof}

\noindent Following example illustrates the Theorem \ref{thm14}.
\begin{example}
 Let $R=(\mathbb{C}[X,Y,Z]/(X^2+Y^2+Z^2))_{(x,y,z)}$ and $I=J=\m=(x,y,z)$. Here 
$x,y,z$ denote image of $X,Y,Z$ in $R$. Since $G(\m)=\bigoplus_{n \geq 0}\m^n/\m^{n+1}
=R$ is reduced, $\ov{\m^n}=\m^n$ for all $n$. Therefore 
$\ov{G}(\m)=\bigoplus_{n \geq 0}\ov{\m^n}/\ov{\m^{n+1}}=G(\m)=R$. Hence 
\begin{eqnarray*}
 H(\ov{G}(\m),t)&=&\frac{1+t}{(1-t)^2}.
\end{eqnarray*}
Therefore for $n \geq 0$,
$$\lm(\m^{n-1}/\m^{n})=2n-1.$$
Since $(x,y)\m=\m^2,$ the normal joint reduction number of $\m$ and $\m$ is zero 
with respect to $(x,y)$ and 
$\ov{r}(\m)\leq 1$. Hence by \cite[Corollary 3.26]{marleythesis}, 
$\ov{e}_2(\m)=0$. Therefore $\ov{G}(\m)$ is Cohen-Macaulay and hence 
by \cite[Theorem 2.3]{viet} $\ov{R}(\m)$ is Cohen-Macaulay. By 
Theorem \ref{thm10}, $\ov{R}(\m,\m)$ is Cohen-Macaulay. Also 
$\ov{e}_2(\m^2)=\ov{e}_2(\m)=0$.
\end{example}

\end{document}